\documentclass[reqno]{amsart}

\usepackage{amssymb,amsmath,amsxtra,graphicx}
\usepackage{times}
\usepackage{bbold}
\usepackage{stmaryrd}
\usepackage{array}
\usepackage{mathrsfs}
\usepackage{dsfont}
\usepackage{url}
\usepackage{tikz}
\usetikzlibrary{shapes,fit,arrows,calc,backgrounds,positioning,matrix,tqft}
\usepackage[T1]{fontenc}

\theoremstyle{plain}
\newtheorem{thm}{Theorem}
\newtheorem{lem}[thm]{Lemma}
\newtheorem{cor}[thm]{Corollary}
\newtheorem*{MT}{Theorem}
\newtheorem*{ack}{Acknowledgment}

\newcommand{\ZZ}{\mathbb{Z}}
\newcommand{\CC}{\mathbb{C}}

\newcommand{\ot}{\otimes}
\newcommand{\ch}{\operatorname{char}}
\newcommand{\T}{\mathsf{T}}
\newcommand{\onto}{\twoheadrightarrow}
\newcommand{\into}{\hookrightarrow}

\renewcommand{\k}{\mathbb{k}}

\newcommand{\End}{\operatorname{End}}
\newcommand{\Ker}{\operatorname{Ker}}
\newcommand{\Id}{\operatorname{Id}}

\newcommand{\cen}{\mathscr{Z}}

\newcommand{\Th}{^\text{\rm th}}
\renewcommand{\d}{\delta}

\newcommand{\e}{\epsilon}

\newcommand{\FD}{{\textnormal{\bf FD}}}

\DeclareMathOperator{\Irr}{Irr}

\newcommand{\fC}{\mathfrak{C}}
\newcommand{\sH}{\mathscr{H}}
\newcommand{\Hcen}{\mathscr{H}\!\cen}
        
\renewcommand{\emph}[1]{{\bfseries\itshape #1}}

\begin{document}

\title[Frobenius divisibility and Hopf centers]
{Frobenius divisibility and Hopf centers}

\author{Adam Jacoby}

\address{Department of Mathematics, Temple University,
    Philadelphia, PA 19122}




\begin{abstract}
A classical theorem of I. Schur states that the degree of any irreducible 
complex representation of a finite group $G$ divides the order of $G/\cen G$,
where $\cen G$ is the center $G$. This note discusses similar divisibility results for certain classes of Hopf algebras.
\end{abstract}

\maketitle


\section*{Introduction}

It is a well-known fact, originally proven by Frobenius \cite{Ff96}, that degree of any irreducible 
complex representation of a finite group $G$ divides the order of $G$; equivalently, the  
degrees of all irreducible representations of the group algebra $\CC G$ divide $\dim_{\CC}\CC G$.
In reference to this result, a finite-dimensional algebra $A$ over an arbitrary algebraically
closed base field $\k$ is said to have the ``Frobenius divisibility'' property if the following holds:
\begin{quote}
(\FD)\qquad  the degree of every irreducible representation of $A$ is a 
divisor of $\dim_\k A$.
\end{quote}
A random finite-dimensional $\k$-algebra will of course fail to satisfy \FD. Nonetheless,
motivated by Frobenius' Theorem, Kaplansky \cite{Ki75} stated the conjecture that {\FD}
does hold for all semisimple Hopf algebras over an algebraically closed field $\k$ of characteristic $0$.
This conjecture has remained open for 40 years. 

Returning to group algebras, Schur \cite[Satz VII]{Si04} strengthened Frobenius' Theorem by showing 
that the degree of any irreducible complex representation of a finite group $G$ does in fact
divide the order of the quotient $G/\cen G$, where $\cen G$ is the center $G$. The goal of this short note is to prove a version of Schur's Theorem for Hopf algebras that expands on the work in \cite{Sm02}. We work with finite-dimensional Hopf algebras
over an algebraically closed base field $\k$ of arbitrary
characteristic. A representation of such a Hopf algebra $H$
is given by a $\k$-vector space $V$ and an algebra map $\rho \colon H \to \End_\k(V)$; and 
$V$ is irreducible if and only if $\rho$ is surjective.
We define $\Hcen(V)$ to be the unique largest Hopf subalgebra of $H$ that is contained in 
the subalgebra $\rho^{-1}(\k \Id_V)$ of $H$. In analogy with the center
of a character \cite[2.27]{Im06}, $\Hcen(V)$ will be called the \emph{Hopf center} of $V$.
Since the dimension of any Hopf subalgebra of $H$ divides $\dim_\k H$ by the Nichols-Zoeller Theorem, it follows that
$\frac{\dim_\k H}{\dim_\k \Hcen(V)}$ is an integer. We may now state our result as follows.

\begin{MT}
Let $\fC$ be a class of finite-dimensional Hopf $\k$-algebras that is closed under tensor products 
and under taking (Hopf) homomorphic images. Assume that all $H \in \fC$ satisfy \FD. Then, 
for every $H \in \fC$ and every irreducible representation $V$ of $H$, 
$\dim_\k V$ is a divisor of $\frac{\dim_\k H}{\dim_\k \Hcen(V)}$.
\end{MT}

Taking $\fC$ to be the class of all finite complex group algebras, all of whose members satisfy {\FD} 
by Frobenius' Theorem, we obtain Schur's result. The proof of the theorem, whose main part is an 
adaptation of an argument due to Tate, will be given in Section~\ref{S:Pf}
after deploying a few preliminaries in Section~\ref{S:Prelims}. The final Section~\ref{S:Apps} presents some applications
of the theorem and its method of proof.

The above notation and terminology remains in effect throughout this note.
In particular, $\k$ denotes an algebraically closed field. Our notation concerning
Hopf algebras follows \cite{Ms93} and \cite{Rd12}.

\section{Preliminaries}
\label{S:Prelims}

\subsection{Hopf Commutators}

Given two elements $h$ and $k$ of a Hopf algebra $H$, we define the
\emph{Hopf commutator} $[h,k]$ by
\begin{equation*}
[h,k]=h_{(1)}k_{(1)}S(h_{(2)})S(k_{(2)}).
\end{equation*} 

\begin{lem} 
\label{Com}
Let $K$ and $L$ be Hopf subalgebras of $H$. 
The following conditions are equivalent:
\begin{enumerate}
\renewcommand{\labelenumi}{(\roman{enumi})}
\item $kl = lk$ for all $k \in K$ and $l \in L$;
\item $[l,k] = \e(l)\e(k)$ for all $k \in K$, $l \in L$.
\end{enumerate}
\end{lem}

\begin{proof}
Assuming (i) we compute $[l,k] = l_{(1)}k_{(1)}S(l_{(2)})S(k_{(2)}) = l_{(1)}S(l_{(2)})k_{(1)}S(k_{(2)}) = \e(l)\e(k)$; so (ii) holds. 
Conversely, assuming (ii) we compute $kl=[k_{(1)},l_{(1)}]l_{(2)}k_{(2)}=\e(k_{(1)})\e(l_{(1)})l_{(2)}k_{(2)}=lk$; so (i) holds.
\end{proof}

\subsection{Hopf Centers}

Let $H$ be a finite-dimensional Hopf $\k$-algebra and let $\zeta(H)$ denote the unique the largest Hopf 
subalgebra of $H$ that is contained in the ordinary center, $\cen H$. For any irreducible representation $V$ of $H$,
Schur's Lemma gives the inclusion
\[
\zeta(H) \subseteq \Hcen(V).
\] 
The reverse inclusion holds if $V$ is \emph{inner faithful}, that is,
no nonzero Hopf ideal of $H$ annihilates $V$.

\begin{lem}
\label{InnerFaithful}
Let $H$ be a finite-dimensional Hopf $\k$-algebra and let $V$ be an inner faithful 
representation of $H$. Then 
\[
\sH\cen(V) \subseteq \zeta(H).
\]
\end{lem}

\begin{proof}
Put $K:= \Hcen(V)$. In view of Lemma~\ref{Com}, we need to show that, for all  $h\in H$ and $k\in K$,
\[
[h,k] = \e(h)\e(k).
\]
To this end, consider the representation map $\rho \colon H \to \End_\k(V)$ and the 
$n\Th$ tensor power $V^{\ot n}$, with corresponding
algebra map $\rho^{\otimes n}\circ\Delta^{n-1}:H\rightarrow H^{\ot n}
\to \End_\k(V)^{\ot n} \cong \End_\k(V^{\ot n})$.
Writing $h_{V^{\ot n}} \in \End_\k(V^{\ot n})$ for the image of $h \in H$ under this map,
we claim that
\[
[h,k]_{V^{\ot n}} = \epsilon(h)\epsilon(k)\Id_{V^{\ot n}} \qquad (h \in H, k \in K).
\]
It will then follow that the element $[h,k]\in H$ acts on $\T V = \bigoplus_{n\ge 0} V^{\otimes n}$ as the 
scalar operator $\epsilon(h)\epsilon(k)\Id_{\T V}$. Since inner faithfulness 
of $V$ is equivalent to faithfulness of $\T V$ in the usual sense by \cite{Rm67}, this will yield the
desired conclusion, $[h,k] = \e(h)\e(k)$.

To prove the claim, we proceed by induction on $n$. The base case $n=0$ states the obvious identity
$\e([h,k]) = \e(h)\e(k)$. For the inductive step, note that
\[
[h,k]_{V^{\ot n}} = \rho_V(h_{(1)}k_{(1)}S(h_{(2n)})S(k_{(2n)}))\ot...\ot \rho_V(h_{(n)}k_{(n)}S(h_{(n+1)})S(k_{(n+1)})).
\]
Since $\Delta^{2n-1}(k)\in \Hcen(V)^{\ot 2n}$, we can move $S(k_{(n+1)})$ past $S(h_{(n+1)})$ to 
rewrite the right hand side above in the following form:
\[
\begin{aligned}
\rho_V&(h_{(1)}k_{(1)}S(h_{(2n)})S(k_{(2n)}))\ot...\ot \rho_V(h_{(n)}k_{(n)}S(k_{(n+1)})S(h_{(n+1)}))\\
&= \rho_V(h_{(1)}k_{(1)}S(h_{(2n-2)})S(k_{(2n-2)}))\ot...\ot \rho_V(h_{(n-1)}k_{(n-1)}S(k_{(n)})S(h_{(n)}))\ot \Id_V\\
&= [h,k]_{V^{\ot n-1}}\ot \Id_V = \epsilon(h)\epsilon(k)\Id_{V^{\ot n-1}} \ot \Id_V = \epsilon(h)\epsilon(k)\Id_{V^{\ot n}}\,,
\end{aligned}
\]
where the penultimate equality uses our inductive hypothesis. This completes the proof.
\end{proof}

\section{Proof of the Theorem}
\label{S:Pf} 

We first show that $\dim_\k V$ divides $\frac{\dim_\k H}{\dim_\k \zeta(H)}$.
Consider the representation map $\rho \colon H \to \End_\k(V)$, which is onto, and
note that $V^{\ot n}$ is an irreducible representation of $H^{\ot n}$ for each $n \ge 0$, because $\rho^{\ot n}$ maps
$H^{\ot n}$ onto $\End_\k(V)^{\ot n} \cong \End_\k(V^{\ot n})$. Since $\zeta(H)$ is commutative,
the multiplication map $\mu_n:=m^{\ot{n-1}}|_{\zeta(H)^{\ot n}}:\zeta(H)^{\ot n}\rightarrow \zeta(H)$ is a morphism of 
Hopf algebras. Hence, $\Ker \mu_n$ is a Hopf ideal of $\zeta(H)^{\ot n}$. Furthermore, the
following diagram commutes:
\begin{equation*}
\begin{tikzpicture}[baseline=(current  bounding  box.center),  >=latex, scale=.7,
bij/.style={above,sloped,inner sep=0.5pt}]
\matrix (m) [matrix of math nodes, row sep=3em,
column sep=3em, text height=1.5ex, text depth=0.25ex]
{\k^{\ot n}  & \k \\ \zeta(H)^{\ot n} & \zeta(H) \\};
\draw[->] (m-1-1) edge node[bij] {$\sim$} (m-1-2);
\draw[->>] (m-2-1) edge node[auto] {\scriptsize $\mu_n$} (m-2-2);
\draw[->>] (m-2-1) edge node[left] {\scriptsize $\rho^{\ot n}$} (m-1-1);
\draw[->>] (m-2-2) edge node[right] {\scriptsize $\rho$} (m-1-2);
\end{tikzpicture} 
\end{equation*}
Thus, $\rho^{\ot n}(\Ker\mu_n)=0$ and so $V^{\ot n}$ is an irreducible representation of 
the Hopf algebra $H_n:=H^{\ot n}/(\Ker \mu_n)H^{\ot n}$, which belongs to $\mathfrak{C}$.
Consequently, $\dim_\k V^{\ot n} = (\dim_\k V)^n$ divides $\dim_\k H_n$ by {\FD}.
Moreover, putting $d:= \dim_\k H$ and $\d:= \dim_k \zeta(H)$ for brevity, we know by the Nichols-Zoeller
Theorem that $H^{\ot n}$ is free of rank $(\frac{d}{\d})^n$ as module over $\zeta(H)^{\ot n}$. Therefore,
\[
\dim_\k (\Ker \mu_n)H^{\ot n} = (\dim_\k \Ker\mu_n)\left(\frac{d}{\d}\right)^n 
= (\d^n-\d)\left(\frac{d}{\d}\right)^n = d^n - \frac{d^n}{\d^{n-1}}\,,
\]
and so $\dim_\k H_n = \frac{d^n}{\d^{n-1}}$. We have shown that $(\dim_\k V)^n$ divides $\frac{d^n}{\d^{n-1}}$ 
for all $n$; in other words, the fraction $q:= \frac{d}{\d\,\dim_\k V}$
satisfies $q^n \in \frac{1}{\d}\ZZ$ for all $n$.
It follows that $q$ is integral over $\ZZ$, and hence $q \in \ZZ$, proving that
$\dim_\k V$ divides $\frac{d}{\d}$\,.

To obtain the stronger assertion, that $\dim_\k V$ divides $\frac{\dim_\k H}{\dim_\k \sH\cen(V)}$, 
let $\sH\Ker V$ denote the largest Hopf ideal of $H$ that is contained in the kernel of $\rho$ and
consider the canonical epimorphism $\overline{\phantom{X}} \colon H \onto \overline{H} =H/\sH\Ker V$.
Then $\overline{H} \in \mathfrak{C}$ and $V$ 
can be viewed as an inner-faithful irreducible 
representation of $\overline{H}$. Thus, $\dim_\k V$ divides 
$\frac{\dim_\k(\overline{H})}{\dim_\k \zeta(\overline{H})}$ by the foregoing. Hence, it suffices to show that 
$\frac{\dim_\k(\overline{H})}{\dim_\k \zeta(\overline{H})}$
divides $\frac{\dim_\k H}{\dim_\k \sH\cen(V)}$.   Writing $ \Hcen_{\overline{H}}(V)$ for the Hopf center of $V$,
viewed as a representation of $\overline H$, we have 
\[
\overline{\Hcen(V)}
\subseteq \Hcen_{\overline{H}}(V) = \zeta(\overline{H}),
\]
where the last equality holds by Lemma \ref{InnerFaithful}. 
Thus, the canonical Hopf epimorphism $H\onto \overline{H} \onto \overline{H}/\overline{H}\zeta(\overline{H})^+$ factors 
through the epimorphism $H \onto H/H\sH\cen(V)^+$; note that $\zeta(H)$ and $\Hcen(V)$ are normal Hopf subalgebras 
of $H$. This gives an epimorphism $H/H\Hcen(V)^+ \onto \overline{H}/\overline{H}\zeta(\overline{H})^+$,
and hence a Hopf monomorphism
$(\overline{H}/\overline{H}\zeta(\overline{H})^+)^* \into (H/H\sH\cen(V)^+)^*$. 
The Nichols-Zoeller Theorem now yields the desired conclusion that
$\dim_\k \overline{H}/\overline{H}\zeta(\overline{H})^+ = \frac{\dim_\k(\overline{H})}{\dim_\k \zeta(\overline{H})}$
divides $\dim_\k H/H\sH\cen(V)^+ = \frac{\dim_\k H}{\dim_\k \sH\cen(V)}$\,, finishing the proof.

\section{Some Applications}
\label{S:Apps}

In this section, we assume that $\ch\k = 0$.

\begin{cor}
Let $H$ be a semisimple quasitriangular Hopf algebra and let $V \in\Irr H $.
Then $\dim_\k V$ divides $\dim_\k H/\dim_\k \sH\cen(V)$.
\end{cor}

\begin{proof}
The class of semisimple quasitriangular Hopf $\k$-algebras is closed under tensor products and quotients, 
and semisimple quasitriangular Hopf $\k$-algebras satisfy \textbf{FD} by \cite{EPGS98}.
\end{proof}

\begin{cor}
Let $H$ be a semisimple Hopf $\k$-algebra and let $V \in \Irr H$ be such that $\chi_V\in \cen(H^*)$. 
Then $\dim_\k V$ divides $\frac{\dim_\k H}{\dim_\k\sH\cen(V)}$.
\end{cor}

\begin{proof}
Let $I$ be a Hopf ideal of $H$ with $I\subseteq \sH\Ker(V)$. 
Then, as in the last part of the proof of the theorem, $V$ descends to a 
representation of $\overline{H}=H/I$, and the character $\chi_V$ belongs to the (Hopf) subalgebra
$\overline{H}^*=I^\perp$ of $H^*$. Therefore, $\chi_V\in\cen(\overline{H}^*)$.
Also, viewing $V^{\ot n}$ as a representation of $H^{\ot n}$
as in the first part of the proof of the Theorem, we have 
$\chi_{V^{\ot n}} = \chi_V^{\ot n}\in\cen((H^{\ot n})^*)$. 
Lastly, by \cite{Zs93}, we know that the degree of any central irreducible character of 
a finite-dimensional Hopf algebra must divide the dimension 
of the Hopf algebra. With these observations, the proof of the theorem goes through.
\end{proof}


\begin{ack}
The results of this note are part of my Ph.D. thesis, written under the direction of my adviser, Professor Martin Lorenz. 
He has my sincerest thanks for his guidance and generosity in sharing his insights.
\end{ack}

\renewcommand{\emph}[1]{{\itshape #1}}

\bibliographystyle{amsplain}
\providecommand{\bysame}{\leavevmode\hbox to3em{\hrulefill}\thinspace}
\providecommand{\MR}{\relax\ifhmode\unskip\space\fi MR }
\providecommand{\MRhref}[2]{%
  \href{http://www.ams.org/mathscinet-getitem?mr=#1}{#2}
}
\providecommand{\href}[2]{#2}


\end{document}